\documentclass[a4paper, 10pt]{article}
\usepackage{amssymb}
\usepackage{amsmath}

\usepackage[driverfallback=dvips,colorlinks,bookmarksopen,bookmarksnumbered,citecolor=red,urlcolor=red]{hyperref}
\usepackage{mathrsfs}
\usepackage{graphicx}
\usepackage{subcaption} 
\usepackage{amsthm}
\newtheorem{claim}{Claim}
\newtheorem{theorem}{Theorem}
\newtheorem{remark}{Remark}

\newtheorem{lemma}{Lemma}

\begin{document}

\begin{center}
Nonlinear Analysis: Modelling and Control, Vol. vv, No. nn, YYYY\\
\copyright\ Vilnius University\\[24pt]
\LARGE
\textbf{On bifurcation of a stage-structured single-species model with harvest}\footnote{This research was supported by National Natural Science Foundation of China (grant Nos.~12301643, 12171117) and Natural Science Foundation of Jiangsu Province, China (No. BK20221106).}\\[6pt]
\small
\textbf {Honghua Bin$^{a}$, Yuying Liu$^{b,c}$, Junjie Wei$^{d}$}\\[6pt]
$^{a}$School of Science, Jimei University, \\ Xiamen, Fujian, 361021, China \\[6pt]
$^{b}$School of Mathematics, China University of Mining and Technology, \\ Xuzhou, Jiangsu, 221116, China \\[6pt]
$^{c}$Jiangsu Center for Applied Mathematics at CUMT, \\ Xuzhou, Jiangsu, 221116, China\\[6pt]
$^{d}$School of Science, Harbin Institute of Technology at Weihai,\\ Weihai, Shandong, 264209, China \\[6pt]

Received: date\quad/\quad
Revised: date\quad/\quad
Published online: data
\end{center}

\begin{abstract}
This paper investigates the dynamics of the Nicholson's blowflies equation with stage structure and harvest. By employing the property of Lambert W function, the existence of positive equilibria is obtained. With aid of the distribution of the eigenvalues in the characteristic equation, the local stability of the equilibria and the existence of Hopf bifurcation of the single-species model are obtained. Furthermore, by applying the results due to Bal$\acute{a}$zs I., R$\ddot{o}$st G. (Internat. J. Bifur. Chaos 31(2021):2150071), when the harvest rate is sufficiently small, the direction of the Hopf bifurcations at the first and last bifurcation values are forward and backward, respectively, and the bifurcating periodic solutions are all asymptotically stable. Finally, Numerical simulations are conducted to validate the theoretical conclusions. These results can be seen as the complement of the works of Shu et al. (J. Differential Equations 255 (2013) 2565).

 \vskip 2mm

\textbf{Keywords:}  Nicholson's blowflies equation, stage structure,  Hopf bifurcation, harvest.

\end{abstract}

\nocite{2009ProcDETAp}

\section{Introduction}
As the basic unit of biological community, population plays a very important role in biological research\cite{Allee1931Animal,Lewis2016The}. The blowfly has become one of the most important organisms in the research of population dynamics because of its short growth cycle, rapid reproduction and low cost\cite{Gurney1980}. In 1954, Nicholson introduced a delayed differential equation to model the population of the laboratory blowfly\cite{Nicho1954}:
\begin{equation}
\frac{dx}{dt}=Px(t-\tau)exp(-\alpha x(t-\tau))-\delta x(t),
\end{equation}
where $P>0$ denotes the maximum per capita daily egg production rate, $1/\alpha$ represents the size at which the population reproduces at its maximum rate, $\delta>0$ is per capita daily adult mortality rate and $\tau$ denotes the maturity delay. Since then, many scientists have studied the dynamics of the Nicholson blowflies equation\cite{SoJ1994,SoJ1998,Gomes2006,Newcomb2005, Berezansky2010,Song2016,Liu2021}.

 Recently, Lou et al. \cite{LR} and Ruiz-Herrera et al. \cite{A} have proposed the following single-species model:
\begin{equation}\label{1}
\left\{ \begin{split}
\frac{dI(t)}{dt}&= B(M(t))-B(M(t-\tau))e^{-\mu\tau}-\mu I(t),\\
\frac{dM(t)}{dt}&= B(M(t-\tau))e^{-\mu\tau}-\gamma M(t),
\end{split}\right.
 \end{equation}
where $I(t)$ and $M(t)$ represent the densities of the immature and mature at time $t$, $\mu$ and $\gamma$ are
the death rates of immature and mature, respectively; $B(M(t))$ is the birth rate function
which depends on the population size; $\tau>0$ is the time development duration from egg to adult.
When the birth function was taken as the Ricker's type function $B(M)=pMe^{-a M}$,  (\ref{1}) becomes
\begin{equation}\label{1.2}
\left\{ \begin{split}
\frac{dI(t)}{dt}&= pM(t)e^{-a M(t)}-pM(t-\tau)e^{-a M(t-\tau)}e^{-\mu\tau}-\mu I(t),\\
\frac{dM(t)}{dt}&= pM(t-\tau)e^{-a M(t-\tau)}e^{-\mu\tau}-\gamma M(t).\end{split}\right.
 \end{equation}
 The second equation in (\ref{1.2}) is
\begin{equation}\label{2}
\frac{dM(t)}{dt}=pM(t-\tau)e^{-a M(t-\tau)}e^{-\mu\tau}-\gamma M(t),
\end{equation}
which is known as Nicholson's blowflies equation with age-structure. The global Hopf bifurcation of (\ref{2}) was studied by Shu et al.\cite{Shu}. Ruiz-Herrera et al. showed that Eq.(\ref{2}) exhibits a Hopf bifurcation when the delay $\tau$ varies in \cite{A}.

When $\mu=0$, the model (\ref{2}) becomes the classical Nicholson's blowflies equation:
\begin{equation}\label{p}
\frac{dM(t)}{dt}=pM(t-\tau)e^{-a M(t-\tau)}-\gamma M(t).
\end{equation}
For the model (\ref{p}), the local and global Hopf bifurcations have been studied by Wei and Li \cite{WL}.  Bal'azs and R$\ddot{o}$st \cite{R} have improved the conclusion on properties of Hopf
bifurcation and proved that the Hopf bifurcation at the first bifurcation value is always supercritical for (\ref{p}).

As is known, it's very common for individuals in ecosystems to be harvested by the predator\cite{ruanwei}. However, ecologists have long hypothesized that excessive harvesting may inadvertently destabilize its population dynamics, leading to fluctuations in abundance\cite{Anderson2008}. Furthermore, a central question in ecology remains: what drives the collapse of certain harvested populations\cite{Fryxell2010}? In this paper, we shall explore the impact of harvest on the species' dynamics, a constant harvest term is introduced into model (\ref{2}), thus we have
\begin{equation}\label{m}
\frac{dM(t)}{dt}=pM(t-\tau)e^{-a M(t-\tau)}e^{-\mu\tau}-\gamma M(t)-h,
\end{equation}
where $h\geq 0$ denotes the harvest term, $M(t)$ denotes the density of the mature at time $t$. Other parameters are the same with those in system (\ref{2}). In this paper, we investigate system (\ref{m}).

We should mention that many scientists have analyzed the dynamics of population models with harvest\cite{Brauer1979,Brauer19798,Brauer1981,Myerscough1992,Conover2002,Neubauer2013}. Chen et. al. analyzed a predator-prey system with seasonal prey harvesting in 2023\cite{Chen2013} and performed the bifurcation analysis of the systems. Yang et. al. investigated a diffusive predator–prey model with Michaelis–Menten type harvesting\cite{Yang2018}, and obtained the existence and property of Hopf bifurcation. Feng et. al. formulated a seasonally interactive model between closed and open seasons with Michaelis-Menten type harvesting\cite{Feng2023}, they proved that under certain conditions, the existence of a periodic solution can be guaranteed provided that a closed season of arbitrary positive duration is established. Recently, Xu et al. \cite{Xu2024} explored a Holling-Tanner predator-prey system incorporating constant-yield prey harvesting, demonstrating that the model exhibits complex dynamical behaviors such as saddle-node bifurcation and Hopf bifurcation.

In this paper, we investigate the existence of two positive equilibria with sufficiently small $h$. We obtain that when $h$ is large enough, no positive equilibrium exists in system (\ref{m}), this indicates that the harvest term affects the existence of positive equilibrium. The local stability of the equilibrium is determined by analyzing the distribution of the characteristic equation's roots. Meanwhile, the existence of Hopf bifurcation is investigated by varying the maturation delay $\tau$. It is worth noting that the equilibrium changes with the delay $\tau$. In addition, the bifurcation direction and the stability of the bifurcating periodic solutions are discussed. Besides, the numerical results indicates that the positive equilibrium would undergo the stability switching phenomenon with the varying of $\tau$.

The remainder of this paper is structured as follows. In section \ref{Sec2}, we prove that the nonnegativity of the solutions in the system can not be satisfied. The existence of the equilibrium is studied. In section \ref{sec3}, we investigate the stability of the equilibrium and perform the bifurcation analysis of the system. In section \ref{sec4}, numerical simulations are performed to validate the theoretical analysis, revealing the existence of periodic solutions in the vicinity of Hopf singularities. Finally, the conclusion is given to completes this paper.
\section{Existence of equilibrium }\label{Sec2}
In this section, we will prove that the harvest term $h$ affects the positivity of solutions in Eq.(\ref{m}). For simplicity, we denote the positive initial value in Eq.(\ref{m}) as $M_0=\phi_0(\theta)$, where $\phi_0$ is a continuous function on $[-\tau,0]$ and $\phi_0>0$ for $\theta\in[-\tau,0]$.

We can easily prove that the solution $M(t,\phi_0)$ of the following problem
$$
\begin{cases}
\frac{dM(t)}{dt}=pM(t-\tau)e^{-a M(t-\tau)}e^{-\mu\tau}-\gamma M(t),\\
M_0=\phi_0(t),~~~t\in[-\tau,0],
\end{cases}
$$
satisfies  $M(t, \phi_0)>0$ for all $t\ge 0$. Which means that when there is no harvest, solution of Eq.(\ref{m}) with positive initial value is positive for $t>0$.

 As to the case when $h>0$, the non-negativity of the solutions of Eq.(\ref{m}) can not be maintained. The conclusion is summarized as follows.
\begin{claim}
For any given positive constants $a, \mu, \gamma, p$ and $h>0$,
there exists a positive function $\phi_0\in\mathcal{C}[-\tau,0]$ such that the solution of Eq.(\ref{m}) with  $\phi_0$ as the initial value is negative at $t=\tau$.
\end{claim}
\begin{proof}
For $\forall~t\in[0,\tau]$, the value of $M(t-\tau)=\phi_0(t-\tau)$ is given. Thus the solution of Eq.(\ref{m}) is as follows:
$$M(t)=e^{-\gamma t}\big[\phi_0(0)+\int_0^t{(pe^{-\mu \tau}\phi_0(s-\tau)e^{-a\phi_0(s-\tau)}-h)}e^{-\gamma s}ds\big].$$
When $t=\tau,$ we have
\begin{equation}{\label{Mtau}}
M(\tau)=e^{-\gamma \tau}\big[\phi_0(0)+\int_0^{\tau}{(pe^{-\mu \tau}\phi_0(s-\tau)e^{-a\phi_0(s-\tau)}-h)}e^{-\gamma s}ds\big].
\end{equation}
Assume that $\phi_0(t)\equiv \phi_0(0)$ is a constant function on $[-\tau,0]$, then $M(\tau)$ in Eq.(\ref{Mtau}) satisfies
\begin{equation*}
\begin{split}
M(\tau) &=e^{-\gamma \tau}\big[\phi_0(0)+\frac{1}{\gamma}(pe^{-\mu \tau}\phi_0(0)e^{-a\phi_0(0)}-h)(e^{\gamma \tau}-1)\big]\\
 &<e^{-\gamma \tau}\big[\phi_0(0)+\frac{1}{\gamma}(p\phi_0(0)-h)(e^{\gamma \tau}-1)\big]\\
 &=e^{-\gamma \tau}\big\{[1+\frac{p}{\gamma}(e^{\gamma \tau}-1)]\phi_0(0)-\frac{h}{\gamma}(e^{\gamma \tau}-1)\big\}.
\end{split}
\end{equation*}
When the initial value satisfies $\phi_0(t)\equiv\phi_0(0)\le \dfrac{h(e^{\gamma \tau}-1)}{\gamma+p(e^{\gamma \tau}-1)}$, then $M(\tau)<0$. This completes the proof.
\end{proof}
\begin{claim}
 Eq.(\ref{m}) with $h>0$ has no positive equilibrium when $p>\gamma$ and $\tau\geq \frac{1}{\mu}\ln\frac{p}{\gamma}$.
\end{claim}
\begin{proof}
Denote
$$
g=pMe^{-a M}e^{-\mu\tau}-\gamma M-h.
$$
Since $\frac{\partial g}{\partial h}=-1<0$, $\frac{\partial g}{\partial\tau}=-\mu pMe^{-aM}e^{-\mu\tau}<0$
 and $g|_{(\tau,h)=(\frac{1}{\mu}\ln\frac{p}{\gamma},0)}<0$ for $M>0$, we have that $g(\tau,h)<0$ for
$\tau>\frac{1}{\mu}\ln\frac{p}{\gamma}$, $h>0$ and $M>0$. Therefore, the claim holds.
\end{proof}

Building upon the above claim, we make the following assumption
$$
(H_1).~~~~~~~~~p>\gamma,~~ \tau<\frac{1}{\mu}\ln\frac{p}{\gamma}.
$$
Clearly, under assumption $(H_1)$ and $h=0$, the model (\ref{m}) admit two equilibria $M_*=\frac{1}{a}(\ln\frac{p}{\gamma}-\mu\tau)>0$ and
$M_0=0$. The following theorem establishes sufficient conditions under which system (\ref{m}) admits two positive equilibria.
\begin{theorem}\label{hlarge}
If $(H_1)$ is satisfied, then there exists $h^*>0$ such that\\
(i) two positive equilibria exist in system (\ref{m}) for $h<h^*$, \\
(ii) a unique positive equilibrium exists in system (\ref{m}) for $h=h^*$, \\
(iii) no positive equilibrium exists in system (\ref{m}) for $h>h^*$,\\
where
\begin{equation}\label{hW}
\begin{split}
&h^*=pe^{-\mu\tau}\overline{M}e^{-a\overline{M}}-\gamma \overline{M},\\
&\overline{M}=\frac{1}{a}[1-W(\frac{\gamma}{p}e^{1+\mu\tau})],
\end{split}
\end{equation}
and $W(\cdot)$ is Lambert W function.
\end{theorem}
\begin{proof}
Denote
$$
g_1(M)=pe^{-\mu\tau}Me^{-aM},~~g_2(M)=\gamma M+h.
$$
Then $g=g_1-g_2$. From $g|_{M=0}=-h<0$, $\lim_{M\rightarrow\infty}g=-\infty$ and $g'(M)|_{M=0}=pe^{-\mu\tau}-\gamma>0$, then
there exists $\overline{M}\in (0,\infty)$ such that $g(\overline{M})$ is a local maximum of $g$.
Then $g'_1(\overline{M})=g'_2(\overline{M})$, and hence,
$$
pe^{-\mu\tau}e^{-a\overline{M}}(1-a\overline{M})=\gamma.
$$
That is
\begin{equation}\label{Mbar}
(1-a\overline{M})e^{1-a\overline{M}}=\frac{\gamma}{p}e^{1+\mu\tau}.
\end{equation}
The assumption $(H_1)$ implies that $\frac{\gamma}{p}e^{1+\mu\tau}<1$. Hence, with aid of the property of Lambert W function,

We thereby establish that Eq.(\ref{Mbar}) admits a unique solution, explicitly defined by $\overline{M}=\frac{1}{a}[1-W(\frac{\gamma}{p}e^{1+\mu\tau})],$ which implies that $\overline{M}$ is the unique extremum of $g$, and $g(\overline{M})$ is
the maximum of $g$ on $(0,\infty)$. Hence, from $g|_{M=0}=-h<0$, $\lim_{M\rightarrow\infty}g=-\infty$ and $g'(M)|_{M=0}=pe^{-\mu\tau}-\gamma>0$, it follows that $g'(M)>0$ for $M\in(0,\overline{M})$, and $g'(M)<0$ for $M\in(\overline{M},\infty)$. Meanwhile, $g(\overline{M})=0$ if and only if $h=h^*$
defined as in (\ref{hW}). Thus, $g(\overline{M})>0$ when $h\in[0,h^*)$, and $g(\overline{M})<0$ when $h\in(h^*,\infty)$. Therefore,
we can draw the conclusion that two positive equilibria exist in system (\ref{m}) when $h\in[0,h^*)$, as is shown in Fig.\ref{fig1}(a), a unique positive equilibrium exists in system (\ref{m}) when $h=h^*$, as is shown in Fig.\ref{fig1}(b), and no positive equilibrium exists in system (\ref{m}) when $h>h^*$, as is shown in Fig.\ref{fig1}(c).
\end{proof}

Figure 1 displays different number of positive equilibria of system (\ref{m}) under different conditions. The parameters are chosen as $\mu=0.1, p=2, a=0.1, \gamma=1, \tau=1$, with $h=0.8$ in (a),  $h=1.030116$ in (b) and $h=1.2$ in (c). Figure 2 shows the equilibria of system (\ref{m}) for $h\in[0,1.2]$, the green curve represents the positive equilibrium $M_*(h)$, the red curve represents the positive equilibrium $M_0(h)$. Other parameters in Figure 2 are the same as those in Figure 1.
\begin{figure}[htbp]
  \centering
  \begin{subfigure}[b]{0.32\textwidth} 
    \includegraphics[width=\textwidth]{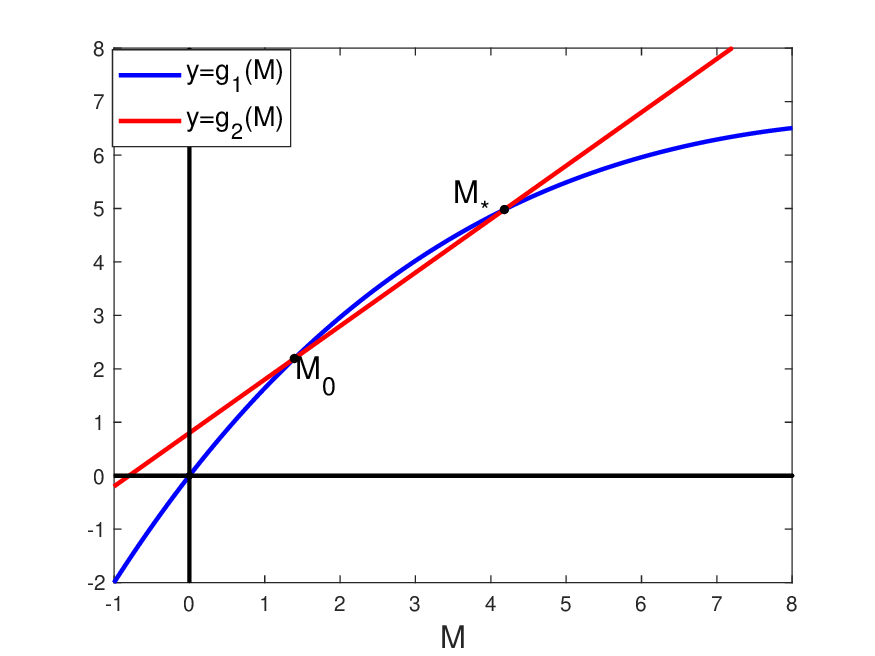}
    \caption{} 
    \label{fig:1a}
  \end{subfigure}
  \begin{subfigure}[b]{0.32\textwidth}
    \includegraphics[width=\textwidth]{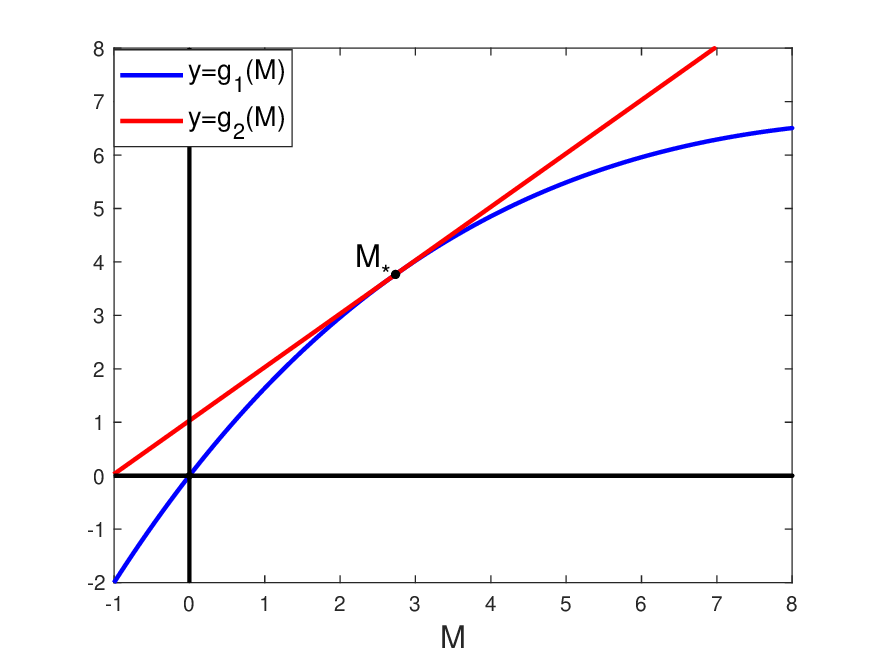}
    \caption{}
    \label{fig:1b}
  \end{subfigure}
  \begin{subfigure}[b]{0.32\textwidth}
    \includegraphics[width=\textwidth]{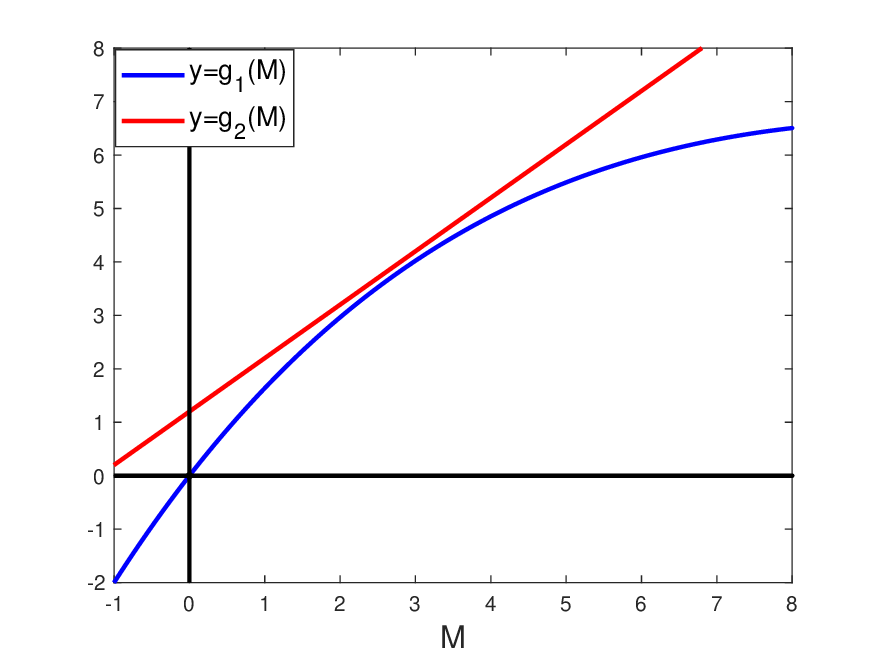}
    \caption{}
    \label{fig:1c}
  \end{subfigure}
 \caption{$(a)$ Two positive equilibria in system (\ref{m}), $(b)$ a unique positive equilibrium in system (\ref{m}), $(c)$  no positive equilibrium in system (\ref{m}).}
\label{fig1}
\end{figure}

\begin{figure}[!ht]
\centering
\begin{tabular}{ccc}
\includegraphics[height=5.2cm,width=7.5cm]{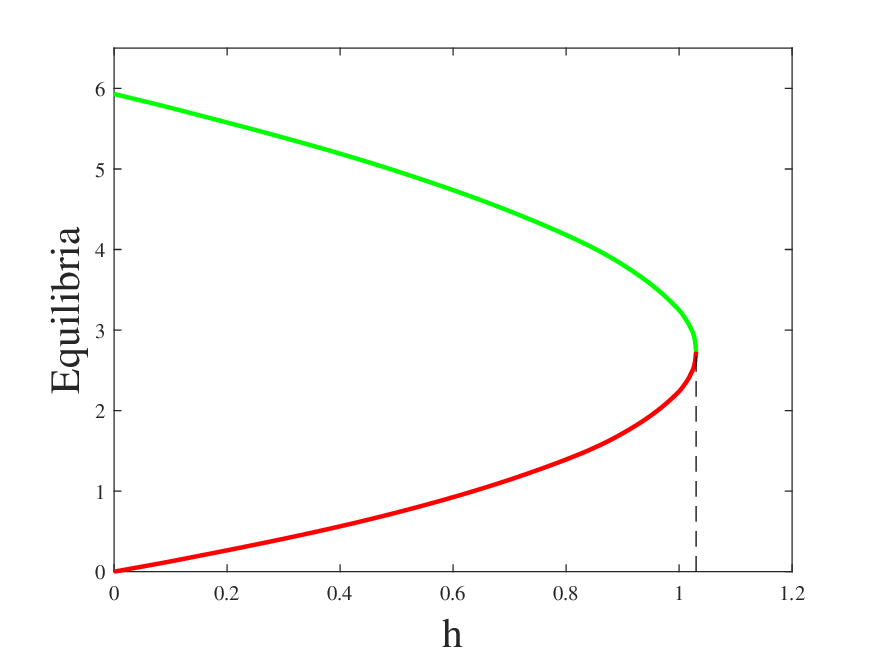}
\end{tabular}
\caption{The diagram of equilibria in system (\ref{m}) for $h>0$.}
\label{fig2}
\end{figure}

\begin{remark}\label{rem1}
According to Theorem \ref{hlarge}, for $h\in [0, h_*)$, system (\ref{m}) has two positive equilibria. We denote the equilibria of system (\ref{m}) on $[0,h_*)$ as $M_*(h)$ and $M_0(h)$.
The proof of Theorem \ref{hlarge} ensures that $M_0(h)<\overline{M}<M_*(h)$. By the implicit function theorem, we know that $M_*(h)$ and $M_0(h)$ are both $C^1$ with respect to $h$. Notice that $M_*$ and $M_0$ are also functions of $\tau$, it can be proved that $M_*$ and $M_0$ are both $C^1$ with respect to $\tau$.
\end{remark}

\section{Stability and Bifurcation Analysis}\label{sec3}

Building upon the analysis in the preceding section, we rigorously establish that under assumption $(H_1)$, for $h\in[0,h^*)$, system (\ref{m}) sustains two equilibria, says
$M_0$ and $M_*$, with $M_0<\overline{M}<M_*$. In this section, we shall study the stability of the equilibria and the existence of Hopf bifurcation at $M_*$.
\subsection{Stability of equilibria and existence of Hopf bifurcation}


To analyze the stability of $M_0$, we derive the following theorem.
\begin{theorem}
If $(H_1)$ and $h<h^*$ are satisfied,  $M_0$ is unstable.
\end{theorem}
\begin{proof}
By theorem \ref{hlarge}, it follows that $M_0$ exists. The linearized equation of system (\ref{m}) at $M=M_0$ yields the following equation
\begin{equation}\label{mlinear0}
\frac{dM(t)}{dt}=pe^{-\mu\tau}e^{-a M_0}(1-aM_0)M(t-\tau)-\gamma M(t),
\end{equation}
the characteristic equation for $M_0$ is
\begin{equation}\label{M0TZ}
\Delta(\lambda):=\lambda+\gamma-e^{-\lambda\tau}pe^{-\mu\tau}(1-aM_0)e^{-aM_0}=0.
\end{equation}
By the Proof of Theorem \ref{hlarge} we have $g'(M_0)=pe^{-\mu\tau}e^{-aM_0}(1-aM_0)-\gamma>0$,
which implies that
$$
\Delta(0)=\gamma-pe^{-\mu\tau}(1-aM_0)e^{-aM_0}<0.
$$
Obviously, $\lim_{\lambda\to+\infty}{\Delta(\lambda)}=+\infty.$ Thus, there exists $\tilde{\lambda}>0$ such that $\Delta(\tilde{\lambda})=0$, which means that Eq.(\ref{M0TZ}) always has a positive root. Therefore, $M_0$ is unstable.
\end{proof}

Then, the local stability of $M_*$ will be investigated. Clearly, the linearization of (\ref{m}) at $M=M_*$ is given by
\begin{equation}\label{3}
\frac{dM(t)}{dt}=pe^{-\mu\tau}e^{-aM_*}(1-aM_*)M(t-\tau)-\gamma M(t),
\end{equation}
and the characteristic equation is
\begin{equation}\label{4}
\lambda+\gamma-e^{-\lambda\tau}pe^{-\mu\tau}e^{-aM_*}(1-aM_*)=0.
\end{equation}
By the Proof in Theorem \ref{hlarge} we have learned that $g'(M_*)=pe^{-\mu\tau}e^{-aM_*}(1-aM_*)-\gamma<0,$ which means
\begin{equation}\label{MsCon}
pe^{-\mu\tau}e^{-aM_*}(1-aM_*)<\gamma.
\end{equation}
Then we have the following theorem.
\begin{theorem}
For the system (\ref{m}), if $(H_1)$ and $h\in [0, h^*)$, then $M_*$ is locally asymptotically stable for $\tau=0$.
\end{theorem}

Following the approach of Beretta and Kuang \cite{Beretta2002}, we analyze the distribution of the roots of (\ref{4}) in the complex plane.
Denote
$$
\lambda+\gamma:=\tilde{P}(\lambda,\tau),~~~-pe^{-\mu\tau}e^{-aM_*}(1-aM_*):=\tilde{Q}(\lambda,\tau),
$$
then the following properties for $\tau\in[0,\frac{1}{\mu}\ln(\frac{p}{\gamma}))$ should be verified.
\begin{enumerate}
\item[(i)]  $\tilde{Q}(0,\tau)+\tilde{P}(0,\tau)\neq 0$;
\item[(ii)] $\tilde{Q}({\rm i}\omega,\tau)+\tilde{P}({\rm i}\omega,\tau)\neq 0$;
\item[(iii)] $\lim_{|\lambda|\to\infty}\sup \left\{\left|\frac{\tilde{Q}(\lambda,\tau)}{\tilde{P}(\lambda,\tau)}\right|:{ Re}\lambda\ge 0\right\}<1$;
\item[(iv)] $F(\omega,\tau)=|\tilde{P}({\rm i}\omega,\tau)|^2-|\tilde{Q}({\rm i}\omega,\tau)|^2$ possesses only a finite number of zeros;
\item[(v)] Each positive zero $\omega(\tau)$ of $F(\omega,\tau)=0$ is differentiable in $\tau$ whenever it exists.
\end{enumerate}
\begin{proof}
From Eq.(\ref{MsCon}), we can deduce that
 $$\tilde{Q}(0,\tau)+\tilde{P}(0,\tau)=\gamma-pe^{-\mu\tau}(1-aM_*)e^{-aM_*}>0,$$
$$\tilde{Q}({\rm i}\omega,\tau)+\tilde{P}({\rm i}\omega,\tau)={\rm i}\omega+\gamma-pe^{-\mu\tau}(1-aM_*)e^{-aM_*}\neq 0.$$
Besides, $\lim_{|\lambda|\to\infty}{|\frac{\tilde{Q}(\lambda,\tau)}{\tilde{P}(\lambda,\tau)}}|=\lim_{|\lambda|\to\infty}|\frac{pe^{-\mu\tau}(1-aM_*)e^{-aM_*}}{\lambda+\gamma}|=0<1$. Thus, properties (i)-(iii) are satisfied.

 We can also get that
 $$
 F(\omega,\tau)=|\tilde{P}({\rm i}\omega,\tau)|^2-|\tilde{Q}({\rm i}\omega,\tau)|^2=\omega^2+\gamma^2-p^2e^{-2\mu\tau}(1-aM_*)^2e^{-2aM_*}
 $$ has at most two zeros as a function of $\omega$. Therefore, property (iv) is satisfied. Besides, by the implicit function theorem, property (v) also holds.
\end{proof}

 Now, by choosing $\tau$ as a varying parameter, we will investigate the existence of Hopf bifurcation at $M_*$ under the assumption $(H_1)$ .

Let ${\rm i}\omega (\omega>0)$ be a purely imaginary root of (\ref{4}), then $\omega$ satisfies
\begin{equation}\label{Omega}
\begin{cases}
\sin\omega\tau=\dfrac{-\omega}{pe^{-\mu\tau}(1-aM_*)e^{-aM_*}},\\
\cos\omega\tau=\dfrac{\gamma}{pe^{-\mu\tau}(1-aM_*)e^{-aM_*}}.
\end{cases}
\end{equation}
We square and sum the two equations in Eq.(\ref{Omega}), it can be concluded that $|\tilde{P}({\rm i}\omega,\tau)|^2-|\tilde{Q}({\rm i}\omega,\tau)|^2=0$. Which means that $\omega$ should be a positive zero of
$$
F(\omega,\tau)=\omega^2+\gamma^2-p^2e^{-2\mu\tau}(1-aM_*)^2e^{-2aM_*}.
$$
 Apparently, $F(\omega,\tau)$ has a positive zero $\omega$ if and only if
\begin{equation}\label{gam2}
\gamma^2<p^2e^{-2\mu\tau}(1-aM_*)^2e^{-2aM_*}.
\end{equation}
Combining Eq.\eqref{MsCon} and \eqref{gam2}, we can easily obtain the fact that  $1-aM_*<0$.

 Define the function
 $$
 I(\tau):=pe^{-\mu\tau}(aM_*-1)e^{-aM_*}-\gamma,
 $$
 we have the following conclusion.

\begin{remark}\label{DI}
If $(H_1)$ and $h<h^*$ are satisfied, then $I(\tau)$ is a decreasing function.
\end{remark}
\begin{proof}
  Since $M_*$ is a positive equilibrium of system \eqref{m},  we have
\begin{equation}\label{Mstar}
pM_*e^{-a M_*}e^{-\mu\tau}=\gamma M_*+h,
\end{equation}
 thus $I(\tau)$ can be rewritten as
 $$
 I(\tau)=(\gamma+\frac{h}{M_*})(aM_*-1)-\gamma.
 $$
 By taking derivative of  $I(\tau)$ with respect to $\tau$, we have
 $$
 I'(\tau)=\frac{hM_*'(\tau)}{M_*^2}+\gamma a M_*'(\tau).
 $$
With aid of Eq.\eqref{Mstar} and the implicit function theorem, it can be proven that $M_*'(\tau)<0$. Thus $I'(\tau)<0$, we obtain that $I(\tau)$ is a decreasing function of $\tau$.

\end{proof}

Denote the set
 \begin{equation}\label{scrI}
   \mathscr{I}=\{\tau\ge 0 : I(\tau)>0\},
 \end{equation}
  Assume that the set $\mathscr{I}\neq \emptyset$, then we get that for  $\forall \tau\in \mathscr{I}$,  a corresponding unique $\omega(\tau)>0$ exists. For any $\tau\in \mathscr{I}$, define $\theta(\tau)\in[0, 2\pi)$ as follows
\begin{equation}\label{Theta}
\begin{cases}
\sin\theta(\tau)=\dfrac{-\omega(\tau)}{pe^{-\mu\tau}(1-aM_*)e^{-aM_*}},\\
\cos\theta(\tau)=\dfrac{\gamma}{pe^{-\mu\tau}(1-aM_*)e^{-aM_*}},
\end{cases}
\end{equation}
then $\theta(\tau)\in (\frac{\pi}{2},\pi) $ by the fact that $1-aM_*<0$. From Eq.\eqref{Theta}, we can draw that $\theta(\tau)$ is unique and well defined for all $\tau \in \mathscr{I}$.

We can also verify that ${\rm i}\omega(\tau) $ is the purely imaginary root of (\ref{4}) when $\tau$ is the zero of the functions $S_n(\tau)$, while
\begin{equation}\label{Sm}
S_n(\tau)=\tau-\dfrac{\theta(\tau)+2n\pi}{\omega(\tau)},~~\tau\in \mathscr{I},~~n\in \mathbb{N}_0.
\end{equation}

From Eq.\eqref{Sm}, we notice that for $\forall\tau\in \mathscr{I}$, $S_n(\tau)>S_{n+1}(\tau)$. As to $S_0(\tau)$, we can obtain the following conclusions.
\begin{lemma}\label{lemma3.1}
 If  $\mathscr{I}$ is nonempty, then the function $S_0(\tau)$ defined in \eqref{Sm} satisfies  $S_0(0)<0$.
\end{lemma}
\begin{proof}
By $\mathscr{I}\neq \emptyset$, there exist $\tau_*\in \mathscr{I}$ so that $I(\tau_*)>0$, and thus $\omega(\tau_*)>0$. If $\tau_*=0$, then $0\in\mathscr{I}$
and $\omega(0)>0$. If $\tau_*>0$, then by Remark \ref{DI}, we have $I(0)>0$. This implies that $\tau=0\in \mathscr{I}$ and $\omega(0)>0$. Thus, by $S_0(\tau)=\tau-\frac{\theta(\tau)}{\omega(\tau)}$ we have $S_0(0)=-\frac{\theta(0)}{\omega(0)}$, combining the fact $\theta(\tau)\in (\frac{\pi}{2},\pi)$ leads to $S_0(0)<0$.
\end{proof}

\begin{lemma}\label{lemma3.2}
 If  $\mathscr{I}\neq \emptyset$,  for $S_0(\tau)$ defined in \eqref{Sm}, a number $\tau_{max}>0$ exists such that $S_0(\tau)\to -\infty$ when $\tau\to \tau_{max}^-$.
\end{lemma}
\begin{proof}
From Remark \ref{DI}, the function $I(\tau)=pe^{-\mu\tau}(aM_*-1)e^{-aM_*}-\gamma$ is a decreasing  function with respect to $\tau$,  the set $\mathscr{I}$ is nonempty, we can deduce that a positive $\tau_{max}\in \partial\mathscr{I}$  and $\tau_{max}\notin \mathscr{I}$, which means
\begin{equation}\label{Ilim}
I(\tau)=pe^{-\mu\tau}(aM_*-1)e^{-aM_*}-\gamma\to 0^+,~~~  \text{when}~\tau\to \tau_{max}^-.
\end{equation}
Combining the second equation of Eq.\eqref{Theta}, and $\theta(\tau)\in (\frac{\pi}{2},\pi) $, we have
$$\lim_{\tau\to \tau_{max}^-}\cos\theta(\tau)=-1^+,$$
thus $\theta(\tau)\to \pi^-$ and $\sin\theta(\tau)\to 0^+$  when $\tau\to \tau_{max}^-$. Combining the equation of $\sin\theta(\tau)$ in Eq.\eqref{Theta} and Eq.\eqref{Ilim}, we have
$$\lim_{\tau\to\tau_{max}^-}\omega(\tau)=0^+.$$
From Eq.\eqref{Sm}, $S_0(\tau)=\tau-\frac{\theta(\tau)}{\omega(\tau)}$, then
\begin{equation*}
\begin{split}
\lim_{\tau\to\tau_{max}^-}S_0(\tau)&=\tau_{max}-\lim_{\tau\to\tau_{max}^-}\frac{\theta(\tau)}{\omega(\tau)}\\
&=-\infty.
\end{split}
\end{equation*}
This concludes the proof.
\end{proof}

\begin{remark}
  We should emphasize that the assumption of $\mathscr{I}$ in Eq.\eqref{scrI} is meaningful, which means that when parameters are given under certain conditions, the set $\mathscr{I}$ is nonempty. For example, when $\frac{p}{\gamma}>e^2$ and $(H_1)$ holds, there exists $\delta>0$, for $h\in[0,\delta)$, we have $\tau\in [0,\delta)\subset \mathscr{I}$. In fact, when $h=\tau=0$, $M_*=\frac{1}{a}\ln\frac{p}{\gamma}$, thus $I(0)=p(aM_*-1)e^{-aM_*}-\gamma=\gamma(\ln\frac{p}{\gamma}-1)-\gamma>0$ when  $\frac{p}{\gamma}>e^2$. Since $M_*$ is continuous with respect to $h$ and $\tau$, thus a positive number $\delta>0$ exists such that $\forall (\tau,h)\in [0,\delta)\times [0,\delta)$, $I(\tau)>0$ holds, and hence the set $\mathscr{I}$ is nonempty.
\end{remark}

In the following, we will introduce a lemma given by Beretta and Kuang \cite{Beretta2002}.

\begin{lemma}\label{lem:Beretta}
Assume that the function $S_n{(\tau)}$ has a simple positive root $\tau^*\in \mathscr{I}$ for some $m\in\mathbb{N}_0$, then a pair of simple purely imaginary roots $\pm {\rm i}\omega_*$ of equation (\ref{4}) exists at $\tau=\tau^*$. Let $\lambda(\tau)=\beta(\tau)+{\rm i}\omega(\tau)$ be the root of (\ref{4}) satisfying $\beta(\tau_*)=0$ and
$\omega(\tau^*)=\omega_*$, then
\begin{equation}\label{eqn:sign}
\begin{array}{l}
\mathrm{Sign}\bigg\{\dfrac{d\beta(\tau)}{d\tau}\bigg\}\bigg|_{\tau=\tau^*} \\
=\mathrm{Sign}\bigg\{\dfrac{\partial F}{\partial \omega}(\omega(\tau^*),\tau^*)\bigg\}\times\mathrm{Sign}\bigg\{\dfrac{dS_n(\tau)}{d\tau}\bigg|_{\tau=\tau^*}\bigg\}.
\end{array}
\end{equation}
Since
$$
\dfrac{\partial F(\omega,\tau)}{\partial \omega}=2\omega>0,
$$
condition (\ref{eqn:sign}) is equivalent to
$$
\delta(\tau^*)=\mathrm{Sign}\bigg\{\dfrac{d\beta(\tau)}{d\tau}\bigg\}\bigg|_{\tau=\tau^*}
=\mathrm{Sign}\big\{\dfrac{dS_{n}(\tau)}{d\tau}\bigg|_{\tau=\tau^*}\big\}.
$$
Therefore, this pair of simple conjugate purely imaginary roots crosses the imaginary axis from left to right if $\delta(\tau^*)=1$ and from right to left if $\delta(\tau^*)=-1$.
\end{lemma}

\begin{remark}
For $m\in \mathbb{N}_0$, denote the set of the zeros of $S_m$ by
$$
J_m=\{\tau^{(m)}|\tau^{(m)} \in \mathscr{I}, S_m(\tau^{(m)})=0 \}.
$$
Clearly, $J_m$ is a finite set, $J_{m_1}\cap J_{m_2}=\emptyset$ for $m_1\neq m_2$, and there exists an integer $N\geq 0$ such the $J_m=\emptyset$ when $m\geq N$.

Rearrange these zeros in the set
$$
J=\bigcup\limits_{m\in \mathbb{N}_0}J_m=\{\tau_0,\tau_1,\cdots\tau_k\},~\text{with}~\tau_j<\tau_{j+1},~0\le j\le k-1.
$$
\end{remark}

We make the following hypothesis:

$$
(H_2)~~~~~~\frac{dS_m(\tau)}{d\tau}(\tau^{(m)})\neq0,~~~\mbox{for}~~~\tau^{(m)}\in J.
$$
Applying Lemmas \ref{lemma3.1} and \ref{lemma3.2}, one can obtain the following conclusion.
\begin{lemma}\label{lem:even}
If $J\neq \emptyset$ and $(H_2)$ is satisfied, then the number of elements of $J$ is even.
\end{lemma}

For simplicity, we denote $\bar{\tau}$ as the supremum of $\tau$ for which $M_*$ exists. Then by applying corollary in \cite{ruanwei}, we can draw the conclusion: if $\tau\in(0,\tau_0)\cup (\tau_k,\bar{\tau})$, then all roots of Eq.(\ref{4}) have negative real parts; When $\tau\in (\tau_0,\tau_k)$, at least one pair of roots have positive real parts of Eq.(\ref{4}). Besides, a pair of purely imaginary roots of Eq.(\ref{4}) exist when $\tau=\tau_n\in J$.

Summarizing the discussions above, we have the following theorem which describe the stability of the equilibrium $M_*$ and the existence of Hopf bifurcation to  Eq.(\ref{m}).
\begin{theorem}\label{theo:bifurcation}
Assume that $(H_1)$ and $h<h_*$ are satisfied.

$(i)$ If either $\mathscr{I}$ is empty or the function $S_0$ has no positive zero in $I(\tau)$, then for all $\tau<\bar{\tau}$, the equilibrium $M_*$ of Eq.(\ref{m}) is locally asymptotically stable;

$(ii)$ If $J\neq \emptyset$ and $(H_2)$ is satisfied, then the equilibrium $M_*$ of Eq.(\ref{m}) is locally asymptotically stable for $\tau\in (0,\tau_0)\cup (\tau_k,\bar{\tau})$ and unstable for $\tau \in (\tau_0,\tau_{k})$ with a Hopf bifurcation occurring at $M_*$ when $\tau=\tau_n\in J$.
\end{theorem}

\subsection{The direction and stability of the Hopf bifurcation}
In this subsection, we shall employ the results due to Bal$\acute{a}$zs and R$\ddot{o}$st \cite{R} to study the direction of the Hopf bifurcations and the stability of the bifurcating periodic solutions.

We introduce the following variable changes:
$$
M(t)=x(t)-M_*~~\mbox{and}~~x(t)=\frac{1}{a}y(\tau t).
$$
Then (\ref{m}) becomes
\begin{equation}\label{h}
y'(t)=-\tau\{\gamma y(t)-p[y(t-1)+aM_*]e^{-\mu\tau}e^{-aM_*}e^{-y(t-1)}\}-\tau a(\gamma M_*+h).
\end{equation}
From the previous subsection, we know that for $h\in[0,h_*)$, if $(H_1)$ and the conditions of $(ii)$ in Theorem \ref{theo:bifurcation}
are satisfied, then (\ref{h}) undergoes a Hopf bifurcation at $y=0$ when $\tau=\tau_n\in J$.
By the Theorem in Faria and Magalh$\tilde{a}$es \cite{Faria}, we know that the flow of (\ref{h}) on the center manifold of the origin
can be given in polar coordinates $(\rho,\xi)$ by the following equation:
$$
\begin{cases}
\dot{\rho}=\alpha\nu'(0)\rho+K\rho^3+\mathcal{O}(\alpha^2\rho+|(\rho,\alpha)|^4),\\
\dot{\xi}=-\omega+\mathcal{O}|(\rho,\alpha)|,
\end{cases}
$$
where $\alpha=\tau-\tau_n$, $\omega=\tau_n\omega(\tau_n)$ and $\nu'(0)=\tau_n\beta'(\tau_n)$. Furthermore, if $\nu'(0)K<0~(\mbox{respectively}~\nu'(0)K>0)$,
the direction of the Hopf bifurcation is forward (respectively, backward). Besides, the bifurcating periodic solutions are stable (respectively, unstable) if
$K<0~(\mbox{respectively}~K>0)$ on the center manifold. Particularly, the stability of bifurcating periodic solutions of (\ref{h}) and that on the center manifold
are consistent at the first critical value $\tau_0$ and last critical value $\tau_k$.

As is known, the equation $pe^{-\mu\tau}e^{-aM_*}=\gamma$ holds when $h=0$, then Eq.(\ref{h}) with $h=0$ can be given by
\begin{equation}\label{y}
y'(t)=-\tau\gamma[y(t)+aM_*(1-e^{-y(t-1)})-y(t-1)e^{-y(t-1)}].
\end{equation}
Notice that the form of (\ref{y}) is
same as (3) in Bal$\acute{a}$zs and R$\ddot{o}$st \cite{R}, Bal$\acute{a}$zs and R$\ddot{o}$st have obtained $K<0$ for $b>1$, see\cite[Subsection 2.2]{R}. For Eq. (\ref{y}),
$b=aM_*-1$. For $h=0$, the existence of purely imaginary root of the characteristic equation (\ref{4}) guarantee $b>1$. In fact, the substitution of ${\rm i}\omega$ into (\ref{4})
leads to $\omega=\sqrt{p^2e^{-2\mu\tau} e^{-2aM_*}b^2-\gamma^2}$. Notice that $M_*=\frac{1}{a}(\ln\frac{p}{\gamma}-\mu\tau)$, we have $\omega=\gamma\sqrt{b^2-1}$. This
implies $|b|>1$. We claim that $b>1$. Otherwise, $b<-1$, that is $b=aM_*-1=\ln\frac{p}{\gamma}-\mu\tau-1<-1$, it follows that $\ln\frac{p}{\gamma}-\mu\tau<0$, which contradicts
to the hypothesis $(H_1)$. This shows that $b>1$. Thus, $K<0$ for Eq. (\ref{y}). Meanwhile, $M_*(h)$ is continuous with respect to $h$, so $K$ does. This implies that, for (\ref{h}),
$K<0$ when $h>0$ and close enough to $0$. Then, we have the following conclusion.
\begin{theorem}\label{direction}
For system (\ref{m}), if $(H_1)$, $(H_2)$, $0\leq h\ll 1$ and $J\neq \emptyset$ are satisfied, then the bifurcating periodic solutions are all asymptotically stable at $\tau_0$ and $\tau_k$, and the direction of the Hopf bifurcation is forward at $\tau_0$, and backward at $\tau_k$.
\end{theorem}
\begin{remark}
For Eq.(\ref{m}) with $h=0$, Shu et al. in \cite{Shu} obtained that the model has only a finite number of Hopf bifurcation singularities, they described how branches of Hopf bifurcations
are paired so the existence of periodic solutions with specific oscillation frequencies occurs only in bounded delay intervals. Our Theorem \ref{direction} shows that $K<0$ for each Hopf bifurcation value $\tau_j$, thus the bifurcating periodic solutions on the center manifold are stable, the direction of the Hopf bifurcation is determined by $\beta'(\tau_j)$. Particularly, we can deduce that when Hopf bifurcations exist in \cite{Shu}, the bifurcating periodic solutions are all stable at the first and the last Hopf bifurcation values. Besides, the bifurcation direction is forward for the first bifurcation and backward for the last bifurcation.
\end{remark}

\section{Numerical simulations }\label{sec4}

In this section, we shall present some numerical simulations for illustration of the theoretical results obtained in the previous section.

The parameters are chosen as
\begin{equation}\label{para}
p=15, \mu=0.2, a=0.2, \gamma=0.1, h=0.1.
\end{equation}
By direct calculation, we can get that $\frac{1}{\mu}\ln(p/\gamma)\approx24.0532.$ With the help of Matlab, one can obtain the equilibria curve $M_0(\tau)$ and $M_*(\tau)$ with the varying of $\tau$, as is shown in Figure \ref{fig3}(a). Obviously, system (\ref{m}) has two positive equilibria for $\tau<14.81$ under the parameters given by Eq.(\ref{para}), and $\bar{\tau}\approx 14.81$. Since $M_*(\tau)$ has already been numerically calculated, we can plot the curve of the function $I(\tau)$, as is shown in Figure \ref{fig3}(b). Obviously, the set $\mathscr{I}=[0, 13.5696)$.
\begin{figure}[htbp]
  \centering
  \begin{subfigure}[b]{0.48\textwidth}
    \includegraphics[width=\textwidth]{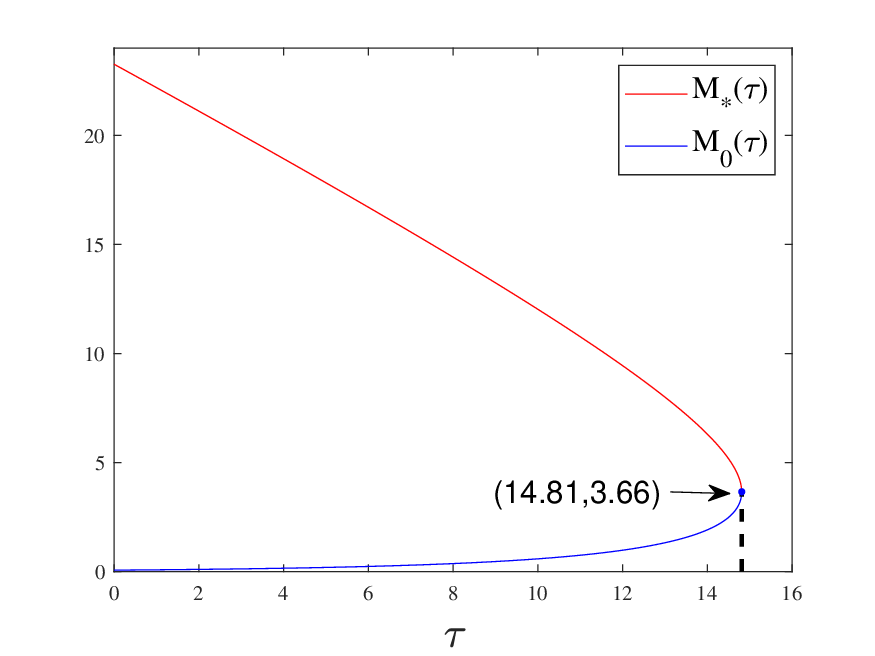}
    \caption{}
  \end{subfigure}
  \hfill
  \begin{subfigure}[b]{0.48\textwidth}
    \includegraphics[width=\textwidth]{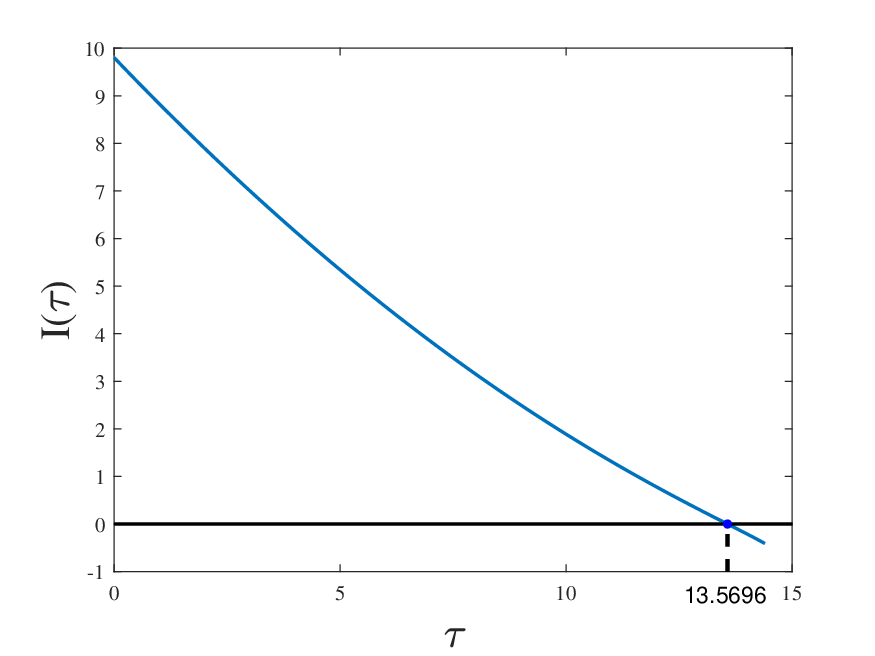}
    \caption{}
  \end{subfigure}
   \caption{(a) The equilibria of system (\ref{m}) with the varying of $\tau$. (b) The function curve of $I(\tau)$ on $(0,14.81)$.}
\label{fig3}
\end{figure}

For $\tau\in\mathscr{I}$, we can obtain the curves of $\sin\theta(\tau)$ and $\cos\theta(\tau)$ in Eq.(\ref{Theta}), and then obtain the curves of $S_n(\tau)$ in Eq.(\ref{Sm}), as is shown in Figure \ref{fig4}. Besides, we can get that $\tau_0\approx 4.53, \tau_1\approx 11.391$, $J=\{\tau_0, \tau_1\}$, $\frac{dS_0(\tau)}{d\tau}|_{\tau=\tau_0}>0$,  $\frac{dS_0(\tau)}{d\tau}|_{\tau=\tau_1}<0$, which means  $\tau=\tau_0$ and $\tau=\tau_1$ are two Hopf singularities.

\begin{figure}[!ht]
\centering
\begin{tabular}{ccc}
\includegraphics[height=6cm,width=8cm]{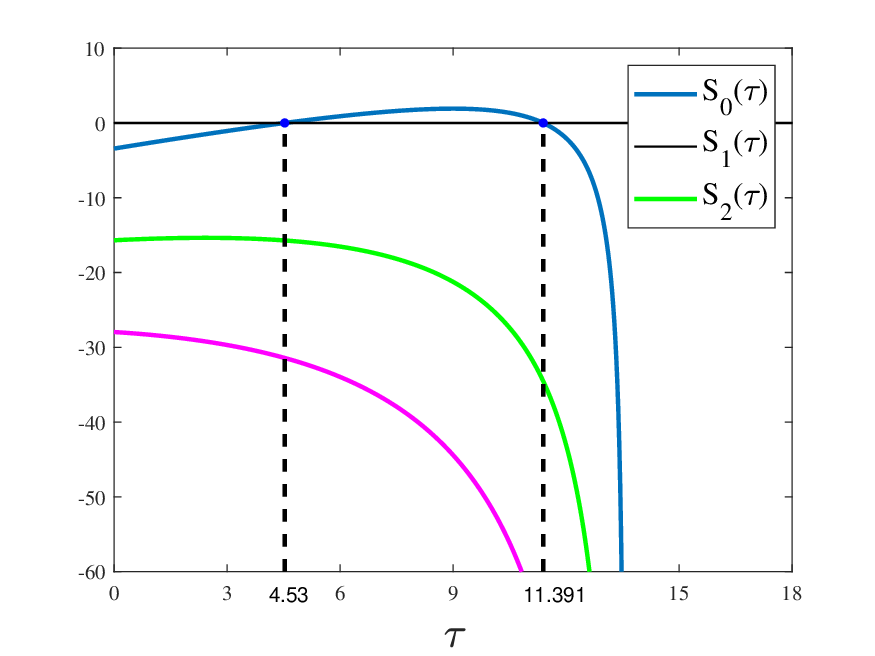}
\\
\end{tabular}
\caption{The graphs of $S_m(\tau)$ on $\mathcal{I}$.}
\label{fig4}
\end{figure}

From Theorem \ref{theo:bifurcation}, we get that $M_*$ is locally asymptotically stable for $\tau\in[0,\tau_0)\cup(\tau_1,\bar{\tau})$ and unstable for $\tau\in(\tau_0,\tau_1)$. The periodic solutions of system (\ref{m}) are bifurcated when $\tau=\tau_0$ and $\tau=\tau_1$.

By varying $\tau$, we have drawn the solution $M(t)$ of system (\ref{m}) under the parameters given by Eq.(\ref{para}), the initial value is  $M_0(t)=10+\cos(2\pi t/\tau)$ for $t\in[-\tau,0]$.
When $\tau=4<\tau_0$, $M_*$ is locally asymptotically stable, which is shown in Figure \ref{fig5}(a); when $\tau=4.7\in(\tau_0,\tau_1)$, periodic solution occurs in system (\ref{m}), this is shown in Figure \ref{fig5}(b);
when $\tau=11\in(\tau_0,\tau_1)$, periodic solution also occurs in system (\ref{m}), this is shown in Figure \ref{fig5}(c); when $\tau=11.8\in(\tau_1, \bar{\tau})$, $M_*$ is locally asymptotically stable, this is shown in Figure \ref{fig5}(d). The numerical simulations is consistent with the analysis in Theorem \ref{theo:bifurcation}.

\begin{figure}[htbp]
  \centering
  \begin{subfigure}[b]{0.48\textwidth}
    \includegraphics[width=\textwidth]{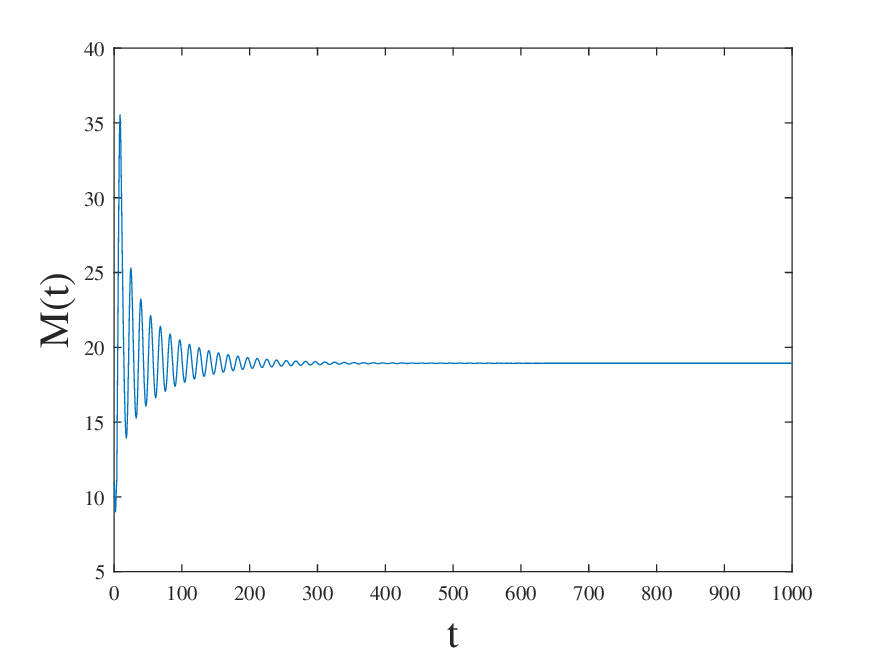}
    \caption{}
  \end{subfigure}
  \hfill
  \begin{subfigure}[b]{0.48\textwidth}
    \includegraphics[width=\textwidth]{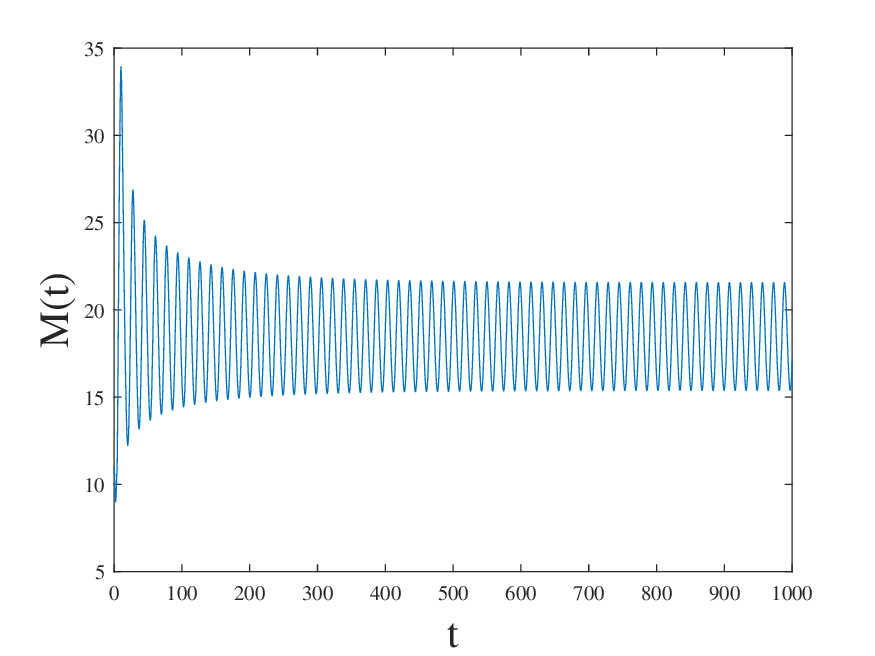}
    \caption{}
  \end{subfigure}
  \\
   \begin{subfigure}[b]{0.48\textwidth}
    \includegraphics[width=\textwidth]{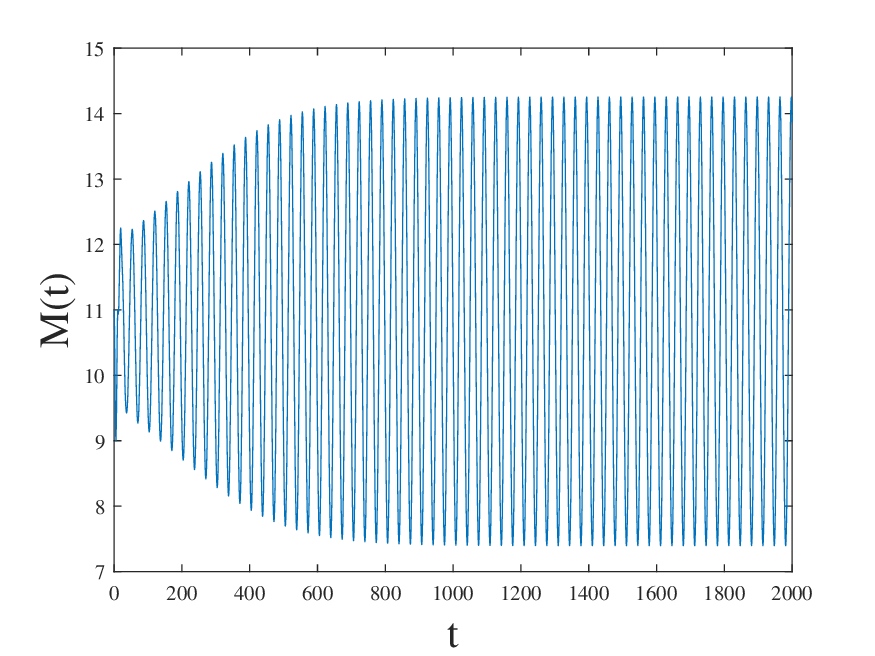}
    \caption{}
  \end{subfigure}
  \hfill
  \begin{subfigure}[b]{0.48\textwidth}
    \includegraphics[width=\textwidth]{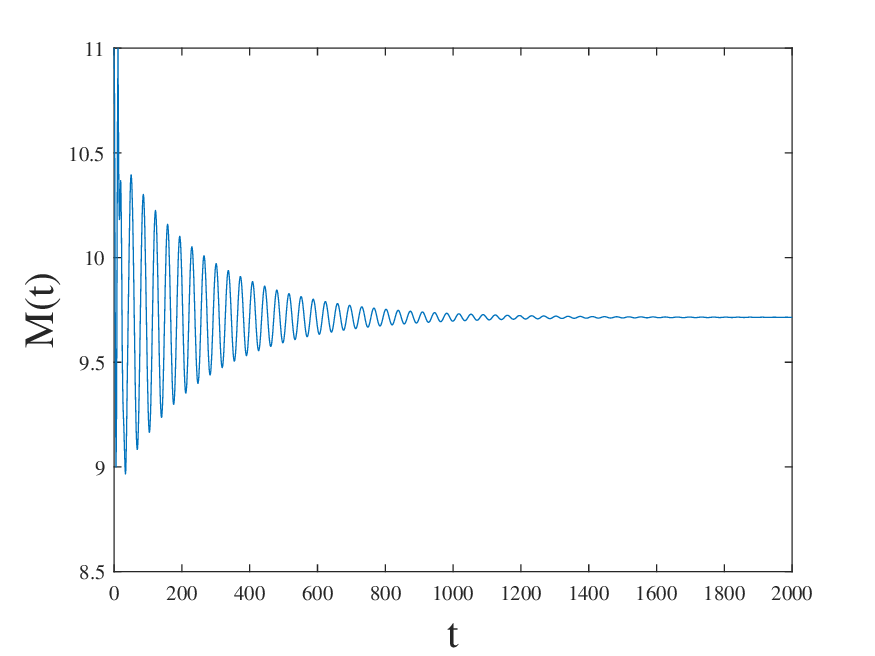}
    \caption{}
  \end{subfigure}
  \caption{The solution $M(t)$ of system (\ref{m}) with different values of $\tau$. (a) $\tau=4$, (b) $\tau=4.7$, (c) $\tau=11$ and (d) $\tau=11.8$. The initial value is $M_0(t)=10+\cos(2\pi t/\tau)$.}
\label{fig5}
\end{figure}

\begin{remark}
Under the data (\ref{para}), there are only two Hopf singularities $\tau_0$ and $\tau_1$, and $\beta'(\tau_0)>0$, $\beta'(\tau_1)<0$. Therefore, we can deduce that when $h\geq 0$ is sufficient small, the bifurcating periodic solutions are all asymptotically stable. Meanwhile, the direction of the Hopf bifurcation is forward at $\tau_0$ and backward at $\tau_1$.
\end{remark}

\section{Conclusion}
In this paper, the dynamics of a single-species model with stage structure and harvest is investigated. Theoretical analysis shows that harvest has an impact on the existence of the equilibrium. When the harvest rate is high, there is no positive equilibrium in system (\ref{m}); when the harvest rate is low enough, there are two positive equilibria in the system. Besides, analysis indicates that when system (\ref{m}) has two positive equilibria, the smaller one is unstable. By varying $\tau$, under certain conditions, we have proved that the larger equilibrium would go from stable to unstable, leading to periodic solutions in the system. It is worth noting that the equilibrium $M_*$ changes with the varying of $\tau$. Finally, numerical simulations are carried out for illustrating the theoretical results.

*\section*{Declaration of competing interest}
 The authors declare that they have no known competing financial interests or personal relationships that could have appeared to influence the work reported in this paper.

\section*{Acknowledgments}
 This research  is supported by the National Natural Science Foundations of China (No. 12301643, 12171117), Natural Science Foundation of Jiangsu Province, China (No. BK20221106).

\end{document}